\documentclass[11pt,a4paper]{article}

\usepackage[T1]{fontenc}
\usepackage[utf8]{inputenc}
\usepackage{lmodern}
\usepackage{amsmath,amssymb,amsthm}
\usepackage[margin=1in]{geometry}
\usepackage{enumitem}
\usepackage{microtype}
\usepackage{needspace}
\usepackage{cite}
\usepackage{booktabs,array}
\usepackage[hidelinks]{hyperref}

\theoremstyle{plain}
\newtheorem{theorem}{Theorem}[section]
\newtheorem{proposition}[theorem]{Proposition}
\newtheorem{lemma}[theorem]{Lemma}
\newtheorem{corollary}[theorem]{Corollary}
\theoremstyle{definition}
\newtheorem{definition}[theorem]{Definition}
\newtheorem{remark}[theorem]{Remark}
\numberwithin{equation}{section}
\DeclareMathOperator{\pos}{pos}
\DeclareMathOperator{\sgn}{sgn}
\newcommand{\KT}{\mathrm{KT}}
\newcommand{\QQ}{\mathbb{Q}}
\newcommand{\one}{\mathbf{1}}
\newcommand{\zero}{\mathbf{0}}
\setlist[enumerate]{label=\textup{(\roman*)},leftmargin=*,itemsep=2pt,topsep=3pt}
\AddToHook{env/theorem/before}{\Needspace{7\baselineskip}}
\title{The Kelly--Trotter conjecture and dimension of poset products\thanks{This work is supported by the
National Natural Science Foundation of China (Grant nos.\ 12471438,
12331016) and the Fundamental Research Funds for the Central Universities
(Grant no.\ GK202501014).}}
\author{Zhaochen Dong\qquad Kaiyun Wang\thanks{Corresponding author.}\\[0.5em]
\small School of Mathematics and Statistics, Shaanxi Normal University\\
\small Xi'an 710119, Shaanxi, P.\ R.\ China\\[0.2em]
\small\texttt{dzc0419@gmail.com}\qquad\texttt{wangkaiyun@snnu.edu.cn}}
\date{}
\hypersetup{
  pdftitle={The Kelly--Trotter conjecture and dimension of poset products},
  pdfauthor={Zhaochen Dong and Kaiyun Wang},
  pdfsubject={Order dimension and Cartesian products of posets},
  pdfkeywords={poset dimension, Cartesian product, covering, incidence poset, linear orders}
}

\begin{document}

\maketitle

\begin{abstract}
We study the order dimension of Cartesian products of finite posets.
Kelly and Trotter conjectured in 1982 that $\dim(P\times Q)\ge\dim P+\dim Q-2$ for all finite posets $P$ and $Q$.
For $m\ge3$, let $R_m$ denote the incidence poset of the complete graph on $m$ vertices.
We prove that there is a constant $C$ such that, for all sufficiently large $m$,
\[
\dim(R_m\times R_m)
\le
\left(1+\frac{2}{\log_2 6}\right)\dim R_m+C.
\]
Since $1+2/\log_2 6<2$, this disproves the Kelly--Trotter conjecture and shows that $2\dim R_m-2-\dim(R_m\times R_m)$
can grow linearly with $\dim R_m$.
For every integer $d\ge8$, we construct an incidence poset $Q_d$ such that $\dim Q_d=d$ and $Q_d$ has the $(3,d)$-covering property.
Consequently,
\[
\dim(P\times Q_d)\le\dim P+d-3
\]
for every poset $P$ with $\dim P\ge3$.
Thus, for every integer $d\ge8$, the poset $Q_d$ violates the Kelly--Trotter conjecture with every poset of dimension at least $3$.
Finally, we prove that a poset $Q$ has an $(r,s)$-covering if and only if $\dim(Q\times C^r)\le s$ for every finite chain $C$ with at least two elements.


\end{abstract}

\section{Introduction}\label{sec:intro}

Order dimension is a fundamental measure of the complexity of partially ordered sets (\emph{posets}, for short).
Let $P=(X,\le_P)$ be a poset.
A \emph{linear extension} of $P$ is a linear order $L$ on $X$ such that $x\le_P y$ implies $x\le_L y$.
A \emph{realizer} of $P$ is a family $L_1,\ldots,L_d$ of linear extensions of $P$ such that
\[
x\le_P y
\quad\Longleftrightarrow\quad
x\le_{L_i}y\ \text{for every }i\in\{1,\ldots,d\}.
\]
The dimension $\dim P$ is the least positive integer $d$ for which $P$ has a realizer of size $d$.
Since it was introduced by Dushnik and Miller~\cite{DM}, dimension has become a central topic in the theory of posets and has developed connections with many areas of combinatorics; see, for example,~\cite{Schnyder,ScottWood,KozikMicekTrotter,Joret,Blake1,Blake2,KnauerTrotter}.

Unless stated otherwise, all posets are finite.
For posets $P$ and $Q$, the following bounds belong to the folklore of dimension theory:
\[
\max\{\dim P,\dim Q\}
\le
\dim(P\times Q)
\le
\dim P+\dim Q.
\]
The upper bound is attained when $P$ and $Q$ are nontrivial chains, but equality need not hold for arbitrary posets.
This raises the question of how large the gap between $\dim P+\dim Q$ and $\dim(P\times Q)$ can be.
In 1982, Kelly and Trotter~\cite[Problem~1.4]{KT} conjectured that
\begin{equation*}\tag{KT}\label{eq:KT}
\dim(P\times Q)\ge\dim P+\dim Q-2
\end{equation*}
for all posets $P$ and $Q$.
For positive integers $n$ and $m$, let $\KT(n,m)$ denote the assertion that~\eqref{eq:KT} holds for every pair of posets $P,Q$ with $\dim P=n$ and $\dim Q=m$.
We have $\KT(n,m)\Longleftrightarrow\KT(m,n)$, and $\KT(n,m)$ holds whenever $\min\{n,m\}\le2$.

In unpublished work, Baker~\cite{Baker} proved that $\dim(P\times Q)=\dim P+\dim Q$ when both factors have distinct least and greatest elements.
Trotter~\cite{Trotter} proved that, for every integer $n\ge3$, there is an $n$-dimensional poset $P$ such that $\dim(P\times P)=2n-2$.
Reuter~\cite{Reuter} proved that $\dim(P\times P)\ge4$ whenever $\dim P=3$.
Felsner, M\"utze, and Wittmann~\cite{FMW} characterized generalized fences $F$ by the condition $\dim(P\times F)\le\dim P+1$ for every poset $P$ and reformulated Reuter's covering property for posets.
They also reduced $\KT(3,3)$ to products of $3$-irreducible posets and settled the cases in which both factors are crowns or neither factor is a crown.
In our earlier work~\cite{DW}, we settled the remaining case, in which exactly one factor is a crown, and thereby proved $\KT(3,3)$.
Bergman~\cite{Bergman1,Bergman2} further studied dimensions of poset products and several questions related to the conjecture~\eqref{eq:KT}.
Despite this progress, the conjecture~\eqref{eq:KT} has remained open for more than four decades.

In this paper, we disprove~\eqref{eq:KT}.
For $m\ge3$, let $R_m$ denote the incidence poset of the complete graph on $m$ vertices.
Our first result is the following.

\begingroup
\renewcommand{\thetheorem}{\ref*{thm:main}}
\begin{theorem}
There is a constant $C$ such that, for all sufficiently large integers $m$,
\[
\dim(R_m\times R_m)\le\left(1+\frac{2}{\log_2 6}\right)\dim R_m+C.
\]
In particular, $\KT(\dim R_m,\dim R_m)$ fails for all sufficiently large $m$.
\end{theorem}
\endgroup

Since $\dim R_m\to\infty$ as $m\to\infty$ and $1+2/\log_2 6<2$, Theorem~\ref{thm:main} shows that the deficit $2\dim R_m-2-\dim(R_m\times R_m)$ can grow linearly with $\dim R_m$.

On the other hand, $\KT(3,3)$ holds.
This raises the question of for which pairs of dimensions $\KT(n,m)$ fails.
Our second result shows that $\KT(n,d)$ fails for every $n\ge3$ and $d\ge8$.
In fact, we prove the stronger statement that, for every integer $d\ge8$, the poset $Q_d=R_{\lambda(d-1)+1}$,
where $\lambda(k)$ denotes the number of maximal intersecting families of subsets of a $k$-element set, has dimension $d$ and violates~\eqref{eq:KT} with every poset of dimension at least three.

\begingroup
\renewcommand{\thetheorem}{\ref*{thm:uniform-counterexamples}}
\begin{theorem}
For every integer $d\ge8$, the poset $Q_d$ has dimension $d$ and has the $(3,d)$-covering property.
Moreover,
\[
\dim(P\times Q_d)\le\dim P+d-3
\]
for every poset $P$ with $\dim P\ge3$.
\end{theorem}
\endgroup

As a consequence, a poset $P$ satisfies
$\dim(P\times Q)\ge\dim P+\dim Q-2$ for every poset $Q$ if and only if $\dim P\le2$.
Indeed, the assertion is immediate when $\dim P\le2$, while for $\dim P\ge3$ the poset $Q_8$ in Theorem~\ref{thm:uniform-counterexamples} provides a counterexample.
Thus $\KT(n,d)$ fails for all integers $n\ge3$ and $d\ge8$, and Table~\ref{tab:KT-status} summarizes these results and the cases not settled here.

Finally, we characterize the $(r,s)$-covering property in terms of dimensions of products with finite grids.
For an integer $n\ge1$, let $[n]=\{1,\ldots,n\}$ with its usual order, and let $[n]^r$ denote the Cartesian product of $r$ copies of $[n]$.
\begingroup
\renewcommand{\thetheorem}{\ref*{thm:characterization}}
\begin{theorem}
Let $Q$ be a poset and let $r,s$ be positive integers.
There is a computable integer $N=N(|Q|,r,s)\ge2$ such that the following conditions are equivalent:
\begin{enumerate}
\item
$Q$ has an $(r,s)$-covering;
\item
$\dim(Q\times[N]^r)\le s$;
\item
$\dim(Q\times[n]^r)\le s$ for every integer $n\ge2$;
\item
$\dim(P\times Q)\le\dim P+s-r$ for every poset $P$ with $\dim P\ge r$.
\end{enumerate}
\end{theorem}
\endgroup

\section{Coverings and incidence posets}\label{sec:embedding}

The Cartesian product of posets $P_i=(X_i,\le_i)$, $i\in[n]$, is the poset on $X_1\times\cdots\times X_n$ ordered coordinatewise.
A \emph{grid} is a finite product of finite chains.
An \emph{order embedding} $f:P\to Q$ is a map such that $x\le_P y$ if and only if $f(x)\le_Q f(y)$.

For a linear order $L$ on a finite set $X$, define
\[
\pos_L(x)=|\{y\in X:y\le_L x\}|.
\]
If $L_1,\ldots,L_d$ is a realizer of a poset $P=(X,\le_P)$, then $x\longmapsto\bigl(\pos_{L_1}(x),\ldots,\pos_{L_d}(x)\bigr)$ is an order embedding of $P$ into a product of $d$ chains.
Conversely, an order embedding of $P$ into a product of $d$ chains implies $\dim P\le d$.

Following Felsner, M\"utze, and Wittmann~\cite[Definition~20]{FMW}, we use the following formulation of Reuter's covering property~\cite{Reuter}.

\begin{definition}\label{def:covering}
Let $r,s$ be positive integers and $P=(X,\le_P)$ be a poset.
We say that $P$ has the \emph{$(r,s)$-covering property} if $P$ admits a realizer $L_1,\ldots,L_s$ such that it is possible to break the linear extensions into consecutive blocks and color these blocks with colors from $[r]$ such that:
\begin{enumerate}
\item If $x\nleq_P y$, then there is some $L_i$ such that $y$ is before $x$ in $L_i$ and the two elements are in distinct blocks.
\item If $x\leq_P y$, then for each $j\in[r]$, there is some $L_i$ with a block of color $j$ containing $x$ and $y$.
\end{enumerate}
\end{definition}

Equivalently, $P$ has the $(r,s)$-covering property if and only if there are finite chains $C_1,\ldots,C_s$, an order embedding $f=(f_1,\ldots,f_s):P\to C_1\times\cdots\times C_s$, and colorings $\gamma_i:C_i\to[r]$, $i\in[s]$, such that for every $x\leq_P y$ and $j\in[r]$, there is some $k\in[s]$ with $f_k(x)=f_k(y)$ and $\gamma_k(f_k(x))=j$. The following proposition follows from \cite[Theorem~10']{Reuter}; see also~\cite[Theorem~21]{FMW}.

\begin{proposition}\label{prop:cover-bound}
If $Q$ has the $(r,s)$-covering property, then
\[
\dim(P\times Q)\le\max\{\dim P,r\}+s-r
\]
for every poset $P$.
\end{proposition}

\begin{proof}
For $\dim P\ge r$, this is~\cite[Theorem~21]{FMW}.
If $\dim P<r$, choose an integer $N\ge2$ such that $P$ embeds into $[N]^r$.
Since $\dim([N]^r)=r$, we have  $\dim(P\times Q)\le\dim([N]^r\times Q)\le s$.
\end{proof}

For a set $S$ and an integer $n \ge 1$, let $\binom{S}{n}$ denote the family of all $n$-element subsets of $S$.
For $m\ge3$, recall that $R_m=\binom{[m]}1\cup\binom{[m]}2$ is ordered by inclusion.
Dushnik proved that $\dim R_m\le s$ if and only if there are linear orders $L_1,\ldots,L_s$ on $[m]$ such that for distinct $x,y,z\in[m]$, some $L_i$ satisfies $y,z<_{L_i}x$; see~\cite[Theorem~III]{Dushnik} and~\cite[Proposition~1.2]{HM}.
To capture the covering property, we add a coloring condition to Dushnik's characterization.

\begin{lemma}\label{lem:vertex-colors}
Let $m,r,s$ be positive integers with $m\ge3$.
Let $L_1,\ldots,L_s$ be linear orders on $[m]$, and let $c_i:[m]\to[r]$ be a coloring for each $i\in[s]$.
Suppose that:
\begin{enumerate}
\item\label{cond:vertex-triples} for distinct $x,y,z\in[m]$, there exists $k\in[s]$ such that $y,z<_{L_k}x$;
\item\label{cond:vertex-colors} for distinct $x,y\in[m]$ and every $j\in[r]$, there exists $k\in[s]$ such that $y<_{L_k}x$ and $c_k(x)=j$.
\end{enumerate}
Then $R_m$ has the $(r,s)$-covering property.
\end{lemma}

\begin{proof}
For each $i\in[s]$, define $f_i:R_m\to[m]$ by
$f_i(X)=\max_{x\in X}\pos_{L_i}(x)$ and define $\gamma_i:[m]\to[r]$ by
$\gamma_i(\pos_{L_i}(x))=c_i(x)$.
Since $\pos_{L_i}$ is a bijection, $\gamma_i$ is well defined.

If $X\subseteq Y$, then $f_i(X)\le f_i(Y)$ for every $i\in[s]$.
Conversely, suppose that $X\nsubseteq Y$, and let $x\in X\setminus Y$.
Since $Y$ has at most two elements, there are distinct $y,z\in[m]\setminus\{x\}$ such that $Y\subseteq\{y,z\}$.
By (i), there exists $k\in[s]$ such that $y,z<_{L_k}x$.
Hence $f_k(Y)<\pos_{L_k}(x)\le f_k(X)$.
Therefore
\[
X\subseteq Y
\quad\Longleftrightarrow\quad
f_i(X)\le f_i(Y)\ \text{for every }i\in[s],
\]
so $f=(f_1,\ldots,f_s):R_m\to[m]^s$ is an order embedding.

It remains to verify the covering property.
Let $X\subseteq Y$ in $R_m$ and  $j\in[r]$.
Take $x\in X$ and $y\in[m]\setminus\{x\}$ such that $Y\subseteq\{x,y\}$.
By (ii), there exists $k\in[s]$ such that $y<_{L_k}x$ and $c_k(x)=j$.
Since $x\in X\subseteq Y\subseteq\{x,y\}$, we have
$f_k(X)=\pos_{L_k}(x)=f_k(Y)$ and $\gamma_k(f_k(X))=c_k(x)=j$.
Thus $R_m$ has the $(r,s)$-covering property.
\end{proof}

For a positive integer $n$, a family $\mathcal F$ of subsets of $[n]$ is \emph{intersecting} if $A\cap B\ne\varnothing$ for all $A,B\in\mathcal F$.
Let $\lambda(n)$ denote the number of maximal intersecting families on $[n]$.
The dimension of $R_m$ is determined by $\lambda(n)$ as follows~\cite{HM,Janzer}.

\begin{lemma}\label{lem:dimension-Rm}
For every $m\ge3$, $\dim R_m=\min\{n:\lambda(n)\ge m\}$.
\end{lemma}

Since $\lambda(n)\le2^{2^n}$, we have $\dim R_m\ge\log_2\log_2m$.
Thus $\dim R_m\to\infty$ as $m\to\infty$.

\section{Counterexamples to the Kelly--Trotter conjecture}\label{sec:examples}

In this section, we disprove~\eqref{eq:KT}.
More precisely, we show that
\[
2\dim R_m-2-\dim(R_m\times R_m)
\]
grows linearly with $\dim R_m$.

Let $m\ge3$ and let $d=\dim R_m$.
Suppose that $R_m$ has an $(r,s)$-covering with $r\le d$.
By Proposition~\ref{prop:cover-bound},
\[
\dim(R_m\times R_m)\le d+s-r.
\]
Thus $\KT(d,d)$ fails whenever $s-r<d-2$.

To construct such coverings, we apply Lemma~\ref{lem:vertex-colors}.
Following an idea of Spencer~\cite[p.~351]{Spencer}, we assign a distinct $0$--$1$ vector in $\{0,1\}^{\ell}$ to each element of $[m]$ and use the rows of an $s\times\ell$ $0$--$1$ matrix to define the linear orders $L_1,\ldots,L_s$.
Such an assignment is possible whenever $m\le2^\ell$.

\begin{lemma}\label{lem:zero-one-matrix}
Let $r,s,\ell$ be positive integers.
Let $B=(b_{ip})_{s\times\ell}$ be a $0$--$1$ matrix, and let $c:[s]\to[r]$ be a coloring of its rows.
Suppose that
\begin{enumerate}
\item $\{(b_{ip},b_{iq}):i\in[s]\}=\{0,1\}^2$ for all distinct $p,q\in[\ell]$;
\item $\{b_{ip}:i\in c^{-1}(j)\}=\{0,1\}$ for every $p\in[\ell]$ and $j\in[r]$.
\end{enumerate}
Then $R_m$ has the $(r,s)$-covering property for every $3\le m\le2^\ell$.
\end{lemma}

\begin{proof}
Let $3\le m\le2^\ell$.
Choose an injection $\varphi:[m]\to\{0,1\}^\ell$, and write $\varphi(x)=(x_1,\ldots,x_\ell)$ for $x\in[m]$.
For distinct $x,y\in[m]$, let $p(x,y)=\min\{p\in[\ell]:x_p\ne y_p\}$.
For each $i\in[s]$, let $L_i$ be the lexicographic order on $[m]$ defined by
\[
x<_{L_i}y
\quad\Longleftrightarrow\quad
x_{p(x,y)}=1-b_{i,p(x,y)}
\]
for distinct $x,y\in[m]$, and define $c_i:[m]\to[r]$ by $c_i(x)=c(i)$ for every $x\in[m]$.

Let $x,y,z\in[m]$ be distinct, and let $p=p(x,y)$ and $q=p(x,z)$.
If $p=q$, then $y_p=z_p=1-x_p$.
By (ii), applied with $j=1$, there is an $i\in c^{-1}(1)$ such that $b_{ip}=x_p$, and hence $y,z<_{L_i}x$.
If $p\ne q$, then by (i), there is an $i\in[s]$ such that $(b_{ip},b_{iq})=(x_p,x_q)$, and again $y,z<_{L_i}x$.

Let $x,y\in[m]$ be distinct, and let $j\in[r]$.
By (ii), there is an $i\in c^{-1}(j)$ such that $b_{i,p(x,y)}=x_{p(x,y)}$.
It follows from the definition of $L_i$ that $y<_{L_i}x$, and $c_i(x)=j$.

The linear orders $L_1,\ldots,L_s$ and the colorings $c_1,\ldots,c_s$ satisfy the conditions of Lemma~\ref{lem:vertex-colors}.
Thus $R_m$ has the $(r,s)$-covering property.
\end{proof}

We now apply Lemma~\ref{lem:zero-one-matrix} with $s=3r$.
For each integer $r\ge2$, define
\[
\ell_r=3^{r-1}\left(\binom{r}{\lfloor r/2\rfloor}+\binom{r-1}{\lfloor r/2\rfloor-1}\right).
\]

\begin{proposition}\label{prop:triple-covering}
For every integer $r\ge2$ and every $3\le m\le2^{\ell_r}$, the poset $R_m$ has the $(r,3r)$-covering property.
Moreover,
\[
\ell_r=\Theta\left(\frac{6^r}{\sqrt{r}}\right).
\]
\end{proposition}

\begin{proof}
Let $t=\lfloor r/2\rfloor$, and let $I_j=\{3j-2,3j-1,3j\}$ for each $j\in[r]$.
Let
\[
\mathcal F_r=\{A\subseteq[3r]:1\in A,\ |A|=r+t,\text{ and }1\le|A\cap I_j|\le2\text{ for every }j\in[r]\}.
\]
For every $A\in\mathcal F_r$, we have
\[
\bigl|\{j\in[r]:|A\cap I_j|=2\}\bigr|=t.
\]
If $|A\cap I_1|=1$, then $A\cap I_1=\{1\}$, and there are $3^{r-1}\binom{r-1}{t}$ choices for $A$.
If $|A\cap I_1|=2$, then the second element of $A\cap I_1$ can be chosen in two ways, and there are $2\cdot3^{r-1}\binom{r-1}{t-1}$ choices for $A$.
Consequently,
\[
|\mathcal F_r|=3^{r-1}\left(\binom{r-1}{t}+2\binom{r-1}{t-1}\right)=3^{r-1}\left(\binom{r}{t}+\binom{r-1}{t-1}\right)=\ell_r,
\]
and thus we write $\mathcal F_r=\{A_1,\ldots,A_{\ell_r}\}$.
Define the $0$--$1$ matrix $B=(b_{ip})_{3r\times\ell_r}$ by $b_{ip}=1$ if and only if $i\in A_p$.
Define $c:[3r]\to[r]$ by $c(i)=j$ whenever $i\in I_j$.

Let $p,q\in[\ell_r]$ be distinct.
Since $1\in A_p\cap A_q$, we have $(b_{1p},b_{1q})=(1,1)$.
Since $A_p$ and $A_q$ are distinct and have the same cardinality, both $A_p\setminus A_q$ and $A_q\setminus A_p$ are nonempty.
Choose $u\in A_p\setminus A_q$ and $v\in A_q\setminus A_p$.
Then $(b_{up},b_{uq})=(1,0)$ and $(b_{vp},b_{vq})=(0,1)$.
Moreover, since $1\in A_p\cap A_q$ and $2t\le r$, we have $|A_p\cup A_q|\le2(r+t)-1\le3r-1$.
Then there exists $w\in[3r]\setminus(A_p\cup A_q)$.
By the definition of $B$, we have $(b_{wp},b_{wq})=(0,0)$.
Therefore,
\[
\{(b_{ip},b_{iq}):i\in[3r]\}=\{0,1\}^2.
\]

Let $p\in[\ell_r]$ and $j\in[r]$.
Since $1\le|A_p\cap I_j|\le2$ and $|I_j|=3$, both $A_p\cap I_j$ and $I_j\setminus A_p$ are nonempty.
By the definition of $B$ and the fact that $c^{-1}(j)=I_j$, we have
\[
\{b_{ip}:i\in c^{-1}(j)\}=\{0,1\}.
\]
Thus $B$ and $c$ satisfy the conditions of Lemma~\ref{lem:zero-one-matrix}, and so $R_m$ has the $(r,3r)$-covering property for every $3\le m\le2^{\ell_r}$.

Finally, since $\binom{r-1}{t-1}=\frac{t}{r}\binom{r}{t}$, we have
\[
\ell_r=3^{r-1}\left(1+\frac{t}{r}\right)\binom{r}{t}.
\]
Since $t=\lfloor r/2\rfloor$ and $\binom{r}{t}=\Theta(2^r/\sqrt{r})$, we obtain $\ell_r=\Theta(6^r/\sqrt{r})$.
\end{proof}

Following~\cite{BMMV}, for a positive integer $n$, let $\alpha(n)$ denote the number of antichains of the Boolean lattice $2^{[n]}$.
By~\cite[Lemma~2.2]{BMMV}, we have $\lambda(n)<\alpha(n-1)$ for every $n\ge2$.

Hansel~\cite{Hansel} proved that
\begin{equation}\label{eq:hansel}
\alpha(n)\le3^{\binom{n}{\lfloor n/2\rfloor}}
\end{equation}
for every positive integer $n$; see also~\cite{Tsai}.
By Lemma~\ref{lem:dimension-Rm}, we have $m\le\lambda(d)<\alpha(d-1)$.
Applying~\eqref{eq:hansel} with $n=d-1$ and using Stirling's formula, we obtain
\begin{equation}\label{eq:incidence-dimension-bound}
\log_2m\le(\log_2 3)\binom{d-1}{\lfloor(d-1)/2\rfloor}=O\left(\frac{2^d}{\sqrt d}\right).
\end{equation}

\begin{theorem}\label{thm:main}
There is a constant $C$ such that, for all sufficiently large integers $m$,
\[
\dim(R_m\times R_m)\le\left(1+\frac{2}{\log_2 6}\right)\dim R_m+C.
\]
In particular, $\KT(\dim R_m,\dim R_m)$ fails for all sufficiently large $m$.
\end{theorem}

\begin{proof}
Let $d=\dim R_m$.
By the observation following Lemma~\ref{lem:dimension-Rm}, we have $d\to\infty$ as $m\to\infty$.
By Proposition~\ref{prop:triple-covering}, there is a constant $a>0$ such that
\[
\ell_r\ge a\frac{6^r}{\sqrt{r}}
\]
for every $r\ge2$.
By~\eqref{eq:incidence-dimension-bound}, there is a constant $b>0$ such that
\[
\log_2m\le b\frac{2^d}{\sqrt d}
\]
for all sufficiently large $m$.
Choose a positive integer $C_0$ such that $a6^{C_0}\ge b$, and let
\[
r=\left\lceil\frac{d}{\log_2 6}\right\rceil+C_0.
\]
Since $1/\log_2 6<1$ and $d\to\infty$, we have $2\le r\le d$ for all sufficiently large $m$.
By the definition of $r$, we have $6^r\ge6^{C_0}2^d$.
Since $r\le d$, we also have $1/\sqrt{r}\ge1/\sqrt d$.
Therefore,
\[
\ell_r\ge a\frac{6^r}{\sqrt{r}}\ge a6^{C_0}\frac{2^d}{\sqrt d}\ge b\frac{2^d}{\sqrt d}\ge\log_2m.
\]
By Proposition~\ref{prop:triple-covering}, the poset $R_m$ has an $(r,3r)$-covering.
Since $r\le d$, Proposition~\ref{prop:cover-bound} implies that
\[
\dim(R_m\times R_m)\le d+2r\le\left(1+\frac{2}{\log_2 6}\right)d+2C_0+2.
\]
Thus the asserted inequality holds with $C=2C_0+2$.

Since $1+2/\log_2 6<2$ and $d\to\infty$ as $m\to\infty$, we have
\[
\left(1+\frac{2}{\log_2 6}\right)d+C<2d-2
\]
for all sufficiently large $m$.
Thus $\dim(R_m\times R_m)<2d-2$, and hence $\KT(d,d)$ fails.
\end{proof}

\begin{remark}
We explain why Proposition~\ref{prop:triple-covering} uses $s=3r$.
Suppose that $B=(b_{ip})_{s\times\ell}$ and $c:[s]\to[r]$ satisfy the conditions of Lemma~\ref{lem:zero-one-matrix}.
For each $j\in[r]$, let $a_j=|c^{-1}(j)|$, and let $B_j$ be the submatrix of $B$ obtained by retaining the rows indexed by $c^{-1}(j)$.
By the second condition of Lemma~\ref{lem:zero-one-matrix}, every column of $B_j$ contains both $0$ and $1$.
Thus $a_j\ge2$, and $B_j$ has at most $2^{a_j}-2$ distinct columns.

By the first condition of Lemma~\ref{lem:zero-one-matrix}, the columns of $B$ are pairwise distinct.
Since $c^{-1}(1),\ldots,c^{-1}(r)$ partition $[s]$, each column of $B$ is determined by the corresponding columns of $B_1,\ldots,B_r$.
Therefore,
\[
\ell\le\prod_{j=1}^r(2^{a_j}-2).
\]

For every integer $a\ge2$,
\[
2^a-2\le6^{(a-1)/2},
\]
with equality if and only if $a=3$.
Indeed, the cases $a=2,3,4$ are immediate, while for $a\ge5$ the inequality follows from $2^a\le6^{(a-1)/2}$.
Since $\sum_{j=1}^r a_j=s$, it follows that
\[
\log_2\ell
\le\sum_{j=1}^r\log_2(2^{a_j}-2)
\le\frac{\log_2 6}{2}\sum_{j=1}^r(a_j-1)
=\frac{\log_2 6}{2}(s-r),
\]
and thus
\[
s-r\ge\frac{2}{\log_2 6}\log_2\ell.
\]

In the construction of Proposition~\ref{prop:triple-covering}, $s=3r$ and $\ell=\ell_r$.
Since $\ell_r=\Theta(6^r/\sqrt r)$, we have
\[
\log_2\ell_r=r\log_2 6-\frac12\log_2 r+O(1)
\qquad\text{as }r\to\infty.
\]
Therefore,
\[
\lim_{r\to\infty}\frac{2r}{\log_2\ell_r}=\frac{2}{\log_2 6}.
\]
Thus three rows in each color class attain the optimal asymptotic ratio of $s-r$ to $\log_2\ell$ among matrices and colorings satisfying the conditions of Lemma~\ref{lem:zero-one-matrix}.
\end{remark}

\section{Three-color coverings of incidence posets}\label{sec:low}

For every integer $d\ge8$, let $Q_d=R_{\lambda(d-1)+1}$.
In this section, we prove that $\dim Q_d=d$ and that $Q_d$ has the $(3,d)$-covering property.

We first establish $\dim Q_d\ge d$.
The function $\lambda$ is nondecreasing.
Indeed, let $q\ge1$, and let $\mathcal F$ be a maximal intersecting family on $[q]$.
Define
\[
\widehat{\mathcal F}
=
\{A\subseteq[q+1]:A\cap[q]\in\mathcal F\}.
\]
If $A,C\in\widehat{\mathcal F}$, then $A\cap[q],C\cap[q]\in\mathcal F$, so $A\cap C\ne\varnothing$.
Hence $\widehat{\mathcal F}$ is intersecting.
If $A\notin\widehat{\mathcal F}$, then $A\cap[q]\notin\mathcal F$.
By the maximality of $\mathcal F$, some $B\in\mathcal F$ is disjoint from $A\cap[q]$.
Since $B\subseteq[q]$, we have $B\in\widehat{\mathcal F}$ and $A\cap B=\varnothing$.
Thus $\widehat{\mathcal F}$ is maximal intersecting.
Moreover,
\[
\mathcal F=\widehat{\mathcal F}\cap2^{[q]},
\]
so $\mathcal F\mapsto\widehat{\mathcal F}$ is injective.
Therefore, $\lambda(q)\le\lambda(q+1)$ for every integer $q\ge1$.
For every positive integer $q<d$, we now have
\[
\lambda(q)\le\lambda(d-1)<\lambda(d-1)+1.
\]
By Lemma~\ref{lem:dimension-Rm}, we have
\[
\dim Q_d
=
\dim R_{\lambda(d-1)+1}
=
\min\{q:\lambda(q)\ge\lambda(d-1)+1\}
\ge d.
\]

If $Q_d$ has a $(3,d)$-covering, then the $d$ linear extensions in the covering form a realizer of $Q_d$.
Hence $\dim Q_d\le d$, and thus $\dim Q_d=d$.
It follows from Proposition~\ref{prop:cover-bound} that for every poset $P$ with $\dim P\ge3$,
\[
\dim(P\times Q_d)
\le \dim P+d-3
<\dim P+\dim Q_d-2.
\]
In particular, $\KT(n,d)$ fails for every integer $n\ge3$.
It remains to construct a $(3,d)$-covering of $Q_d$.

\subsection{The case \texorpdfstring{$d\ge10$}{d >= 10}}\label{subsec:all-d}

Let $d\ge10$, let $t=\lfloor d/2\rfloor$, and define $a_j=\lfloor(d+j-1)/3\rfloor$ for $j\in[3]$.
Since $a_1+a_2+a_3=d$, choose a partition $I_1,I_2,I_3$ of $[d]$ such that $|I_j|=a_j$ for every $j\in[3]$ and $1\in I_1$.
Define
\[
\mathcal A_d=\left\{A\in\binom{[d]}{t}:1\in A,\quad 0<|A\cap I_j|<a_j\text{ for every }j\in[3]\right\}.
\]

\begin{lemma}\label{lem:all-d-count}
For every $d\ge10$, we have $|\mathcal A_d|>\log_2 \alpha(d-2)$, and hence $2^{|\mathcal A_d|}>\lambda(d-1)$.
\end{lemma}

\begin{proof}
Writing $[x^t]p(x)$ for the coefficient of $x^t$ in $p(x)$, we have
\[
|\mathcal A_d|=[x^t]\bigl(x(1+x)^{a_1-1}-x^{a_1}\bigr)
\prod_{j=2}^3\bigl((1+x)^{a_j}-1-x^{a_j}\bigr).
\]
Let $\beta_d=\binom{d-2}{\lfloor(d-2)/2\rfloor}$.
Table~\ref{tab:high-counts} lists the values of $|\mathcal A_d|$ and $\beta_d$ for $10\le d\le20$.

\begin{table}[htbp]
\centering
\caption{Values of $|\mathcal A_d|$ and $\beta_d$ for $10\le d\le20$.}
\label{tab:high-counts}
\begin{tabular}{r c r r}
\toprule
$d$ & $(a_1,a_2,a_3)$ & $|\mathcal A_d|$ & $\beta_d$\\
\midrule
10 & $(3,3,4)$ & 90 & 70\\
11 & $(3,4,4)$ & 164 & 126\\
12 & $(4,4,4)$ & 396 & 252\\
13 & $(4,4,5)$ & 690 & 462\\
14 & $(4,5,5)$ & 1550 & 924\\
15 & $(5,5,5)$ & 2800 & 1716\\
16 & $(5,5,6)$ & 6100 & 3432\\
17 & $(5,6,6)$ & 11008 & 6435\\
18 & $(6,6,6)$ & 23710 & 12870\\
19 & $(6,6,7)$ & 42805 & 24310\\
20 & $(6,7,7)$ & 90895 & 48620\\
\bottomrule
\end{tabular}
\end{table}

An upset of $2^{[q+1]}$ is determined by its two sections on the last coordinate, each of which is an upset of $2^{[q]}$.
Thus $\alpha(q+1)\le \alpha(q)^2$.
Since $\alpha(8)<2^{76}$ by~\cite{Wiedemann}, we obtain $\log_2 \alpha(d-2)<76\cdot2^{d-10}$ for $10\le d\le13$.
The four bounds $76,152,304,608$ are smaller than the corresponding values of $|\mathcal A_d|$ in Table~\ref{tab:high-counts}.
For $14\le d\le20$, Table~\ref{tab:high-counts} shows that $5|\mathcal A_d|>8\beta_d$.
Since $\log_2 3<8/5$, Hansel's bound~\eqref{eq:hansel} yields
\[
|\mathcal A_d|>\frac85\beta_d>(\log_2 3)\beta_d\ge\log_2 \alpha(d-2).
\]

Now let $d\ge21$, and choose uniformly at random a $t$-element subset $X$ of $[d]$ containing $1$.
Let $M_d=\binom{d-1}{t-1}$.
Since $t\le d/2$, the successive conditional probabilities of including prescribed elements other than $1$ are at most $1/2$.
Hence $\Pr(I_1\subseteq X)\le2^{1-a_1}$ and $\Pr(I_j\subseteq X)\le2^{-a_j}$ for $j\in\{2,3\}$.
For $j\in\{2,3\}$, we also have $\Pr(I_j\cap X=\varnothing)\le2^{1-a_j}$.
To see this, write $g=a_j$.
If $d=2t$, then $[d]\setminus X$ is a uniformly chosen $t$-subset of $[d]\setminus\{1\}$, and all factors after the first in the probability of containing $I_j$ are at most $1/2$.
If $d=2t+1$, then $[d]\setminus X$ is a uniformly chosen $(t+1)$-subset of a $2t$-element set.
The first two factors have product $(t+1)/(2(2t-1))\le1/2$, and every subsequent factor is at most $1/2$.
In both cases, the probability is at most $2^{1-g}$.
Since $a_1,a_2,a_3\ge7$, the union bound yields
\[
1-\frac{|\mathcal A_d|}{M_d}
\le2^{1-a_1}+3\cdot2^{-a_2}+3\cdot2^{-a_3}
\le\frac1{16}.
\]
Moreover,
\[
\frac{M_d}{\beta_d}=
\begin{cases}
2(d-1)/d,&d\text{ even},\\
2(d-1)/(d+1),&d\text{ odd},
\end{cases}
\qquad\text{so}\qquad
\frac{M_d}{\beta_d}\ge\frac{20}{11}.
\]
Therefore
\[
|\mathcal A_d|\ge\frac{15}{16}M_d\ge\frac{75}{44}\beta_d
>\frac85\beta_d>(\log_2 3)\beta_d\ge\log_2 \alpha(d-2).
\]
Finally, $\lambda(d-1)<\alpha(d-2)<2^{|\mathcal A_d|}$ by~\cite[Lemma~2.2]{BMMV}.
\end{proof}

\begin{proposition}\label{prop:all-d}
For every $d\ge10$, the poset $Q_d$ has the $(3,d)$-covering property.
\end{proposition}

\begin{proof}
Let $B=(b_{iA})$ have rows indexed by $[d]$ and columns indexed by $\mathcal A_d$, with $b_{iA}=1$ if and only if $i\in A$.
Define $c:[d]\to[3]$ by $c(i)=j$ for $i\in I_j$.
Every column of $B$ is nonconstant on each $I_j$.
For distinct $A,A'\in\mathcal A_d$, the sets $A\cap A'$, $A\setminus A'$, and $A'\setminus A$ are nonempty, since $1\in A\cap A'$ and $|A|=|A'|$.
Also, $|A\cup A'|\le2t-1<d$.
Thus the two columns contain all four pairs in $\{0,1\}^2$.
By Lemma~\ref{lem:all-d-count}, we have $\lambda(d-1)+1\le2^{|\mathcal A_d|}$.
Lemma~\ref{lem:zero-one-matrix} therefore shows that $Q_d$ has a $(3,d)$-covering.
\end{proof}

The same definition at $d=9$ and $d=8$ yields $|\mathcal A_9|=36$ and $|\mathcal A_8|=18$, respectively.
However, $2^{36}<\lambda(8)+1$ and $2^{18}<\lambda(7)+1$ by~\cite[Proposition~1.2]{BMMV}.
Thus these column families are too small for Lemma~\ref{lem:zero-one-matrix}.
For $d\in\{8,9\}$, we instead apply Lemma~\ref{lem:vertex-colors} directly.
For linear orders $(L_i)_{i\in I}$ on a set $W$ and distinct $x,y\in W$, define
\[
D(x,y)=\{i\in I:y<_{L_i}x\}.
\]

\subsection{The case \texorpdfstring{$d=9$}{d = 9}}\label{subsec:nine}

\begin{proposition}\label{prop:nine}
The poset $Q_9=R_{229\,809\,982\,113}$ has the $(3,9)$-covering property.
\end{proposition}

\begin{proof}
Let $U=[9]$, and partition $U$ into $I_1=\{1,2,3\}$, $I_2=\{4,5,6\}$, and $I_3=\{7,8,9\}$.
Define
\[
\begin{aligned}
\mathcal A&=\left\{A\in\binom{U}{4}:1\in A,\quad 1\le|A\cap I_j|\le2\text{ for }j\in[3]\right\},\\
\mathcal B&=\left\{B\in\binom{U}{4}:B\cap I_1\in\{\{2\},\{3\}\},\quad 1\le|B\cap I_j|\le2\text{ for }j\in\{2,3\}\right\}.
\end{aligned}
\]
Counting according to the intersection sizes, we obtain $|\mathcal A|=2\cdot3\cdot3+3\cdot3+3\cdot3=36$ and $|\mathcal B|=2\cdot2\cdot3\cdot3=36$.
Every $B\in\mathcal B$ is disjoint from exactly three members of $\mathcal A$.
Indeed, if its intersection sizes with $(I_1,I_2,I_3)$ are $(1,2,1)$, these members are obtained from $U\setminus B$ by deleting either the point in $(I_1\setminus B)\setminus\{1\}$ or one of the two points in $I_3\setminus B$.
The case $(1,1,2)$ is symmetric.

Let
\[
\mathcal V=2^{\mathcal A}\cup
\bigcup_{B\in\mathcal B}
\left\{\{B\}\cup\mathcal H:\mathcal H\subseteq\mathcal A,\quad A\cap B\ne\varnothing\text{ for every }A\in\mathcal H\right\}.
\]
Every member of $\mathcal V$ is an intersecting family of $4$-subsets of $U$.
The families in this union are pairwise disjoint, so
\[
|\mathcal V|=2^{36}+36\cdot2^{33}=377\,957\,122\,048.
\]
Fix an ordering of $\mathcal A\cup\mathcal B$ and represent each $\mathcal H\in\mathcal V$ by its membership vector.
For each $i\in U$, let $L_i$ be the lexicographic order on $\mathcal V$ in which the order at coordinate $C$ is $0<1$ if $i\in C$ and $1<0$ otherwise.
If $C$ is the first coordinate on which $\mathcal H$ and $\mathcal G$ differ, then
\[
D(\mathcal H,\mathcal G)=
\begin{cases}
C,&C\in\mathcal H,\\
U\setminus C,&C\notin\mathcal H.
\end{cases}
\]
Each of these sets meets every $I_j$.
For fixed $\mathcal H$, any two sets of the first form intersect because they belong to $\mathcal H$.
Any two sets of the second form intersect because they have size five.
For a mixed pair $C,U\setminus E$, we have $C\in\mathcal H$ and $E\notin\mathcal H$, so $C\ne E$.
Since $|C|=|E|=4$, we have $C\setminus E\ne\varnothing$.
Thus $D(\mathcal H,\mathcal G)\cap D(\mathcal H,\mathcal F)\ne\varnothing$ for all distinct $\mathcal H,\mathcal G,\mathcal F\in\mathcal V$.

Define $c_i(\mathcal H)=j$ whenever $i\in I_j$.
The orders and colorings satisfy both conditions of Lemma~\ref{lem:vertex-colors}.
By~\cite[Proposition~1.2]{BMMV},
\[
\lambda(8)+1=229\,809\,982\,113<|\mathcal V|.
\]
Choose a subfamily $\mathcal V'\subseteq\mathcal V$ of size $\lambda(8)+1$ and restrict the orders and colorings to $\mathcal V'$.
Lemma~\ref{lem:vertex-colors} shows that $Q_9$ has a $(3,9)$-covering.
\end{proof}

\subsection{The case \texorpdfstring{$d=8$}{d = 8}}\label{subsec:eight}

Let $U=\{a,b\}\sqcup J\sqcup K$, where $|J|=|K|=3$, and let $X=J\cup K$.
Let
\[
T=\{S\subseteq X:0<|S\cap J|<3,\quad 0<|S\cap K|<3\},
\]
ordered by inclusion.

\begin{lemma}\label{lem:eight-local}
There is a family $\mathcal V$ of $714\,756$ upsets of $T$ and linear orders $L_i$, $i\in U$, on $\mathcal V$ such that $D(H,H')\cap D(H,H'')\ne\varnothing$ for all distinct $H,H',H''\in\mathcal V$.
Moreover, $D(H,H')$ meets each of $\{a,b\}$, $J$, and $K$ for all distinct $H,H'\in\mathcal V$.
\end{lemma}

\begin{proof}
Fix an ordering of $T$ and represent each upset $H$ by its membership vector.
For $S\in T$, define $C(S)=\{a\}\cup S$.
For $i\in U$, define $L_i$ lexicographically, using $0<1$ at coordinate $S$ if $i\in C(S)$, and $1<0$ otherwise.
If $S$ is the first coordinate on which $H$ and $H'$ differ, then
\[
D(H,H')=
\begin{cases}
C(S),&S\in H,\\
U\setminus C(S),&S\notin H.
\end{cases}
\]
Both sets meet $\{a,b\}$, $J$, and $K$.
For fixed $H$, two sets of the first form both contain $a$, and two of the second form both contain $b$.
A mixed pair $C(S),U\setminus C(E)$ can be disjoint only if $S\subseteq E$, which is impossible when $S\in H$ and $E\notin H$ because $H$ is an upset.
Thus the required intersection properties hold for every family of upsets of $T$.

For $q\in\{2,3,4\}$, let $T_q=\{S\in T:|S|=q\}$.
Then $|T_2|=9$, $|T_3|=18$, and $|T_4|=9$.
Let $\mathcal U_+$ be the family of upsets containing $T_4$, and let $\mathcal U_-$ be the family of upsets disjoint from $T_2$.
Let $G$ be the graph on $J\times K$ in which two vertices are adjacent when they agree in one coordinate.
For $W\subseteq V(G)$, let $e(G[W])$ be the number of edges with both endpoints in $W$.
Identify $T_2$ with $V(G)$ and $T_3$ with $E(G)$, so that containment is incidence.
For $H\in\mathcal U_+$, let $W=T_2\setminus H$.
The edges incident with vertices outside $W$ must belong to $H$, while the edges of $G[W]$ can be chosen independently.
Consequently, $|\mathcal U_+|=\sum_{W\subseteq V(G)}2^{e(G[W])}$.
For each $W$, the exponent is the sum of $\binom{|W\cap R|}{2}$ over the three rows and three columns $R$ of $J\times K$.
A direct enumeration of the $2^9$ subsets yields
\[
\begin{aligned}
\sum_{W\subseteq V(G)}z^{e(G[W])}
&=34+54z+81z^2+42z^3+90z^4+36z^5+51z^6\\
&\quad+36z^7+36z^8+6z^9+18z^{10}+18z^{11}+9z^{14}+z^{18}.
\end{aligned}
\]
At $z=2$, this identity yields $|\mathcal U_+|=488\,450$.
The map $H\mapsto\{S\in T:X\setminus S\notin H\}$ is a bijection from $\mathcal U_+$ to $\mathcal U_-$, so $|\mathcal U_-|=488\,450$.
An upset in $\mathcal U_+\cap\mathcal U_-$ is determined by an arbitrary subset of $T_3$.
Thus $\mathcal V=\mathcal U_+\cup\mathcal U_-$ has size $2\cdot488\,450-2^{18}=714\,756$.
\end{proof}

\begin{proposition}\label{prop:eight}
The poset $Q_8=R_{1\,422\,565}$ has the $(3,8)$-covering property.
\end{proposition}

\begin{proof}
Let $\mathcal V_A$ and $\mathcal V_B$ be two disjoint copies of the family in Lemma~\ref{lem:eight-local}.
Define
\[
\begin{array}{lll}
A_1=\{1,2\},&A_2=\{3,5,6\},&A_3=\{4,7,8\},\\
B_1=\{5,6\},&B_2=\{7,1,2\},&B_3=\{8,3,4\}.
\end{array}
\]
Relabel the orders on $\mathcal V_A$ so that the index sets $\{a,b\}$, $J$, and $K$ become $A_1,A_2,A_3$, respectively, and relabel those on $\mathcal V_B$ using $B_1,B_2,B_3$.
For $i\in\{1,2,3,4\}$, let $L_i$ place $\mathcal V_B$ before $\mathcal V_A$; for $i\in\{5,6,7,8\}$, let $L_i$ place $\mathcal V_A$ before $\mathcal V_B$.
Within each part, use its relabelled local order.
For $x\in\mathcal V_A$, define $c_i(x)=j$ when $i\in A_j$; for $x\in\mathcal V_B$, define $c_i(x)=j$ when $i\in B_j$.

Let $x,y,z\in \mathcal V_A\cup\mathcal V_B$ be distinct.
If they lie in the same part, then $D(x,y)\cap D(x,z)\ne\varnothing$ by Lemma~\ref{lem:eight-local}.
If $y,z$ lie in the part not containing $x$, then $D(x,y)=D(x,z)\ne\varnothing$.
In the remaining case, exchange $y$ and $z$ if necessary so that $x,y$ lie in the same part.
If $x,y\in\mathcal V_A$, then $D(x,y)\cap A_1\ne\varnothing$ and $A_1\subseteq D(x,z)$.
If $x,y\in\mathcal V_B$, use $B_1$ in the same way.
Thus condition~\ref{cond:vertex-triples} of Lemma~\ref{lem:vertex-colors} holds.

For $x,y$ in the same part, condition~\ref{cond:vertex-colors} follows from Lemma~\ref{lem:eight-local}.
If $x\in\mathcal V_A$ and $y\in\mathcal V_B$, then $D(x,y)=\{1,2,3,4\}$, which meets every $A_j$.
If $x\in\mathcal V_B$ and $y\in\mathcal V_A$, then $D(x,y)=\{5,6,7,8\}$, which meets every $B_j$.
Hence condition~\ref{cond:vertex-colors} also holds.
Since $|\mathcal V_A\cup\mathcal V_B|=1\,429\,512>1\,422\,565$, choose a subset $\mathcal W\subseteq\mathcal V_A\cup\mathcal V_B$ of size $1\,422\,565$.
Restrict the orders and colorings to $\mathcal W$.
Lemma~\ref{lem:vertex-colors} shows that $Q_8$ has a $(3,8)$-covering.
By~\cite[Proposition~1.2]{BMMV}, $\lambda(7)=1\,422\,564$, and Lemma~\ref{lem:dimension-Rm} now implies $\dim Q_8=8$.
\end{proof}

\begin{theorem}\label{thm:uniform-counterexamples}
For every integer $d\ge8$, the poset $Q_d$ has dimension $d$ and has the $(3,d)$-covering property.
Moreover,
\[
\dim(P\times Q_d)\le\dim P+d-3
\]
for every poset $P$ with $\dim P\ge3$.
\end{theorem}

\begin{proof}
For $d\ge10$, Proposition~\ref{prop:all-d} establishes that $Q_d$ has a $(3,d)$-covering.
The cases $d=9$ and $d=8$ are established by Propositions~\ref{prop:nine} and~\ref{prop:eight}, respectively.
Thus $Q_d$ has a $(3,d)$-covering for every integer $d\ge8$, and the theorem follows from the argument above.
\end{proof}

\begin{corollary}\label{cor:universal-KT}
A poset $P$ satisfies
\[
\dim(P\times Q)\ge\dim P+\dim Q-2
\]
for every poset $Q$ if and only if $\dim P\le2$.
\end{corollary}

\begin{proof}
If $\dim P\le2$, then
\[
\dim(P\times Q)\ge\dim Q\ge\dim P+\dim Q-2
\]
for every poset $Q$.
If $\dim P\ge3$, then Theorem~\ref{thm:uniform-counterexamples} with $d=8$ gives
\[
\dim(P\times Q_8)\le\dim P+5
<\dim P+\dim Q_8-2.
\]
\end{proof}

\begin{table}[htbp]
\centering
\caption{The status of $\KT(n,m)$ for $n,m\ge3$.}
\label{tab:KT-status}
\begin{tabular}{c c c c c c c}
\toprule
$n\backslash m$ & $3$ & $4$ & $5$ & $6$ & $7$ & $m\ge8$\\
\midrule
$3$       & $\checkmark$ & $?$ & $?$ & $?$ & $?$ & $\times$\\
$4$       & $?$ & $?$ & $?$ & $?$ & $?$ & $\times$\\
$5$       & $?$ & $?$ & $?$ & $?$ & $?$ & $\times$\\
$6$       & $?$ & $?$ & $?$ & $?$ & $?$ & $\times$\\
$7$       & $?$ & $?$ & $?$ & $?$ & $?$ & $\times$\\
$n\ge8$  & $\times$ & $\times$ & $\times$ & $\times$ & $\times$ & $\times$\\
\bottomrule
\end{tabular}
\end{table}

In Table~\ref{tab:KT-status}, $\checkmark$ denotes a true assertion, $\times$ denotes a false assertion, and $?$ denotes a case not settled here.
The entry $\KT(3,3)$ follows from~\cite{DW}, and every entry marked $\times$ follows from Theorem~\ref{thm:uniform-counterexamples}.
The table is symmetric because $\KT(n,m)\Longleftrightarrow\KT(m,n)$.
Cases with $\min\{n,m\}\le2$ are true and are not displayed.

\medskip
\noindent\textbf{Problem.}
Determine whether $\KT(n,m)$ holds for $3\le n\le m\le7$ with $(n,m)\ne(3,3)$.

\section{Coverings and finite grids}\label{sec:coverings}

In this section, we prove Theorem~\ref{thm:characterization}.
Fix a poset $Q$ and positive integers $r,s$.
Let $\QQ$ be the chain of rational numbers, and let $\zero=(0,\ldots,0)$ and $\one=(1,\ldots,1)$.
For $a\in\QQ^r$, let $\sgn(a)\in\{-1,0,1\}^r$ be its coordinatewise sign vector.
Linear extensions and realizers of $Q\times\QQ^r$ are understood in the usual sense.

\begin{lemma}\label{lem:ramsey}
For every positive integer $h$, there is a computable integer $N=N(h,r,s)\ge2$ with the following property.
If $|Q|=h$ and $L_1,\ldots,L_s$ is a realizer of $Q\times[N]^r$, then there are sets $A_1,\ldots,A_r\subseteq[N]$, each of size three, such that for $a,b\in A_1\times\cdots\times A_r$, whether $(x,a)<_{L_i}(y,b)$ holds depends only on $i,x,y$, and $\sgn(b-a)$.
\end{lemma}

\begin{proof}
Fix $\varepsilon\in\{-1,0,1\}^r$.
For $j\in[r]$, let $\nu_j=1$ if $\varepsilon_j=0$ and let $\nu_j=2$ otherwise, and let $k_\varepsilon=\sum_{j=1}^r\nu_j$.
For each $S\in\binom{[N]}{k_\varepsilon}$, list the elements of $S$ increasingly and assign the first $\nu_1$ elements to coordinate $1$, the next $\nu_2$ elements to coordinate $2$, and so on.
If $\nu_j=1$, let $a_j=b_j$ be the assigned element.
If $\nu_j=2$, assign the two elements to $a_j,b_j$ so that $\sgn(b_j-a_j)=\varepsilon_j$.
Color $S$ by the truth values of $(x,a)<_{L_i}(y,b)$ for $i\in[s]$ and $x,y\in Q$.
There are at most $c=2^{sh^2}$ colors.

Let $R(k,t,c)$ be the least integer $v$ such that every $c$-coloring of $\binom{[v]}{k}$ has a monochromatic $t$-element subset.
List the vectors in $\{-1,0,1\}^r$ as $\varepsilon_1,\ldots,\varepsilon_M$, where $M=3^r$, and define
\[
N_M=3r,\qquad
N_{\ell-1}=R(k_{\varepsilon_\ell},N_\ell,c)
\quad(\ell=M,M-1,\ldots,1),
\qquad
N=N_0.
\]
By successive applications of the finite Ramsey theorem, there is a set $H\subseteq[N]$ of size $3r$ on which every one of these colorings is constant.
Partition $H$ into sets $A_1,\ldots,A_r$ of size three so that every element of $A_j$ is smaller than every element of $A_{j+1}$.
For $a,b\in A_1\times\cdots\times A_r$, let $\varepsilon=\sgn(b-a)$.
The distinct entries occurring in $a$ and $b$ form the $k_\varepsilon$-element set used in the coloring for $\varepsilon$.
Thus the comparison of $(x,a)$ and $(y,b)$ in $L_i$ depends only on $i,x,y$, and $\sgn(b-a)$.
\end{proof}

For Lemmas~\ref{lem:rational}--\ref{lem:extraction}, let
$N=N(|Q|,r,s)$ be the integer in Lemma~\ref{lem:ramsey}.
Suppose that $\dim(Q\times[N]^r)\le s$, and let
$L_1,\ldots,L_s$ be linear extensions, not necessarily distinct, forming a realizer of $Q\times[N]^r$.
By Lemma~\ref{lem:ramsey}, there are sets $A_1,\ldots,A_r\subseteq[N]$, each of size $3$, with the stated property.
Identify each $A_j$ increasingly with $[3]$, and denote the restrictions of $L_1,\ldots,L_s$ to $Q\times[3]^r$ by the same symbols.

\begin{lemma}\label{lem:rational}
There are linear orders $\widehat L_1,\ldots,\widehat L_s$ forming a realizer of $Q\times\QQ^r$ such that, for every $i\in[s]$, the comparison of $(x,a)$ and $(y,b)$ in $\widehat L_i$ depends only on $x,y$, and $\sgn(b-a)$.
\end{lemma}

\begin{proof}
For distinct $(x,a),(y,b)\in Q\times\QQ^r$, choose $a',b'\in[3]^r$ with $\sgn(b'-a')=\sgn(b-a)$, and define $(x,a)<_{\widehat L_i}(y,b)$ if and only if $(x,a')<_{L_i}(y,b')$.
By Lemma~\ref{lem:ramsey}, this is independent of the choice of $a',b'$.

For two vectors in $\QQ^r$, the values occurring in each coordinate admit an increasing map into $[3]$.
Hence $\widehat L_i$ compares every two distinct elements, extends the strict product order, and the intersection of $\widehat L_1,\ldots,\widehat L_s$ is the strict product order.
For three vectors in $\QQ^r$, the values occurring in each coordinate admit an increasing map into $[3]$ preserving all three pairwise comparisons.
Thus each $\widehat L_i$ is transitive, and $\widehat L_1,\ldots,\widehat L_s$ form a realizer of $Q\times\QQ^r$.
\end{proof}

Fix $i\in[s]$ and write $<$ for $<_{\widehat L_i}$.
For $x\in Q$, call $j\in[r]$ a \emph{principal coordinate} of the fiber $\{x\}\times\QQ^r$ if $(x,a)<(x,b)$ whenever $a_j<b_j$.

\begin{lemma}\label{lem:principal}
Each fiber has a unique principal coordinate.
\end{lemma}

\begin{proof}
Fix $x\in Q$.
For $v\in\QQ^r$, write $v\succ0$ if $(x,\zero)<(x,v)$.
By Lemma~\ref{lem:rational}, we have $(x,a)<(x,b)$ if and only if $b-a\succ0$.
If $u,v\succ0$, then $(x,\zero)<(x,u)<(x,u+v)$, and hence $u+v\succ0$.
Every nonzero vector with nonnegative coordinates is positive.

Let $e_1,\ldots,e_r$ be the standard basis vectors, and define $t\triangleleft j$ if $e_j-e_t\succ0$.
For distinct $t,j\in[r]$, exactly one of $t\triangleleft j$ and $j\triangleleft t$ holds by totality.
If $t\triangleleft j$ and $j\triangleleft k$, then $e_k-e_t=(e_k-e_j)+(e_j-e_t)\succ0$, so $t\triangleleft k$.
Thus $\triangleleft$ is a strict linear order on $[r]$.
Let $j$ be its greatest element.

If $r=1$, then $j$ is principal.
Let $r\ge2$ and let $v\in\QQ^r$ satisfy $v_j>0$.
Write
\[
v=\sum_{t\ne j}\left(\frac{v_j}{r-1}e_j+v_te_t\right).
\]
If $v_t\ge0$, the corresponding summand is positive.
If $v_t<0$, its sign vector is that of $e_j-e_t$, which is positive since $t\triangleleft j$.
Thus every summand is positive, and hence $v\succ0$.
Therefore $j$ is principal.

Suppose that distinct $j,t\in[r]$ are both principal.
Choose $a,b\in\QQ^r$ with $a_j<b_j$ and $a_t>b_t$.
Then $(x,a)<(x,b)$ and $(x,b)<(x,a)$, a contradiction.
\end{proof}

Let $j(x)$ be the principal coordinate of the fiber over $x$, and let $\delta_x(u)=(x,u\one)$ for $u\in\QQ$.
Write $x\ll y$ if every point of the fiber over $x$ precedes every point of the fiber over $y$, and write $x\sim y$ if neither $x\ll y$ nor $y\ll x$ holds.

\begin{lemma}\label{lem:classes}
The relation $\sim$ is an equivalence relation, and $\ll$ linearly orders its classes.
All fibers over a class $E$ have the same principal coordinate $j(E)$.
Moreover,
\begin{equation}\label{eq:class-principal}
x,y\in E,\quad a_{j(E)}<b_{j(E)}
\quad\Longrightarrow\quad
(x,a)<(y,b).
\end{equation}
\end{lemma}

\begin{proof}
If $u<\min_j a_j$ and $v>\max_j a_j$, then $\delta_x(u)<(x,a)<\delta_x(v)$.
Hence $x\ll y$ if and only if $\delta_x(u)<\delta_y(v)$ for all $u,v\in\QQ$.

Suppose that $\delta_y(v)<\delta_x(u)$ for some $u<v$.
By Lemma~\ref{lem:rational}, the same inequality holds for every pair $u<v$.
For $c,d\in\QQ$, choose $u<\min\{c,d\}$ and $v>\max\{c,d\}$.
Then $\delta_y(d)<\delta_y(v)<\delta_x(u)<\delta_x(c)$, so $y\ll x$.
Thus
\begin{equation}\label{eq:diagonal}
x\sim y
\quad\Longleftrightarrow\quad
\delta_x(u)<\delta_y(v)\text{ and }\delta_y(u)<\delta_x(v)
\text{ for all }u<v.
\end{equation}

The relation $\sim$ is reflexive and symmetric.
If $x\sim y$ and $y\sim z$, then $\delta_x(0)<\delta_y(1)<\delta_z(2)$ and $\delta_z(0)<\delta_y(1)<\delta_x(2)$.
By Lemma~\ref{lem:rational}, these inequalities imply~\eqref{eq:diagonal} for $x,z$.
Thus $\sim$ is transitive.

The relation $\ll$ is transitive.
Suppose that $x\sim x'$ and $x\ll y$.
If $x'\sim y$, then $x\sim y$, while $y\ll x'$ would imply $x\ll x'$.
Thus $x'\ll y$.
The same argument for the second class shows that $\ll$ linearly orders the classes.

Let $x\sim y$.
If $a_{j(x)}<b_{j(y)}$, choose $u,v\in\QQ$ with $a_{j(x)}<u<v<b_{j(y)}$.
Then $(x,a)<\delta_x(u)<\delta_y(v)<(y,b)$.
Similarly, $a_{j(x)}>b_{j(y)}$ implies $(y,b)<(x,a)$.

Suppose that $j=j(x)\ne j(y)=t$.
Choose $a,b\in\QQ^r$ with $(a_j,b_j)=(2,3)$, $(a_t,b_t)=(-2,-1)$, and $(a_\ell,b_\ell)=(0,1)$ for $\ell\notin\{j,t\}$.
Since $a_j>b_t$, we have $(y,b)<(x,a)$.
On the other hand, $\sgn(b-a)=\one$, so by~\eqref{eq:diagonal} we have $(x,a)<(y,b)$, a contradiction.
Thus $j(x)=j(y)$, and~\eqref{eq:class-principal} follows.
\end{proof}

\begin{lemma}\label{lem:extraction}
The poset $Q$ has an $(r,s)$-covering.
\end{lemma}

\begin{proof}
For each $i\in[s]$, let $C_i$ be the chain of equivalence classes defined by $\widehat L_i$, let $f_i(x)$ be the class containing $x$, and let $\gamma_i:C_i\to[r]$ assign to each class its principal coordinate.

If $x\le_Q y$, then $f_i(x)\le f_i(y)$ for every $i\in[s]$, since otherwise $(y,\zero)<_{\widehat L_i}(x,\zero)$.
If $x\nleq_Q y$, then $(x,\zero)\nleq(y,\one)$ in $Q\times\QQ^r$, so some $\widehat L_i$ satisfies $(y,\one)<_{\widehat L_i}(x,\zero)$.
Equation~\eqref{eq:class-principal} rules out $f_i(x)=f_i(y)$, while $f_i(x)<f_i(y)$ is impossible by the order of the classes.
Hence $f_i(x)>f_i(y)$.
Therefore $f=(f_1,\ldots,f_s):Q\to C_1\times\cdots\times C_s$ is an order embedding.

Let $x\le_Q y$ and $j\in[r]$, and let $e_j$ be the $j$th standard basis vector.
Since $(x,e_j)\nleq(y,\one-e_j)$, some $\widehat L_i$ satisfies $(y,\one-e_j)<_{\widehat L_i}(x,e_j)$.
As $f_i(x)\le f_i(y)$, we have $f_i(x)=f_i(y)$.
If the principal coordinate of this class were $t\ne j$, then the $t$th coordinates of $e_j$ and $\one-e_j$ would be $0$ and $1$, contradicting~\eqref{eq:class-principal}.
Thus $\gamma_i(f_i(x))=j$.
By Definition~\ref{def:covering}, $Q$ has an $(r,s)$-covering.
\end{proof}

\begin{theorem}\label{thm:characterization}
Let $Q$ be a poset and let $r,s$ be positive integers.
There is a computable integer $N=N(|Q|,r,s)\ge2$ such that the following conditions are equivalent:
\begin{enumerate}
\item\label{char:cover}
$Q$ has an $(r,s)$-covering;
\item\label{char:finite}
$\dim(Q\times[N]^r)\le s$;
\item\label{char:grid}
$\dim(Q\times[n]^r)\le s$ for every integer $n\ge2$;
\item\label{char:uniform}
$\dim(P\times Q)\le\dim P+s-r$ for every poset $P$ with $\dim P\ge r$.
\end{enumerate}
\end{theorem}

\begin{proof}
Let $N=N(|Q|,r,s)$ be the integer in Lemma~\ref{lem:ramsey}.
By Proposition~\ref{prop:cover-bound}, \ref{char:cover} implies \ref{char:uniform}.

Let $n\ge2$.
Since $[n]^r$ is a product of $r$ chains, $\dim([n]^r)\le r$.
For $r\ge2$ and $j\in[r]$, let $a_j\in\{1,2\}^r$ have coordinate $j$ equal to $2$ and all other coordinates equal to $1$, and let $b_j=3\one-a_j$.
Then $a_j\nleq b_j$ for every $j$, while $a_j\le b_t$ whenever $j\ne t$.
No linear extension can place both $b_j$ before $a_j$ and $b_t$ before $a_t$ for distinct $j,t$.
Thus $\dim([n]^r)\ge r$.
The case $r=1$ is immediate, and hence $\dim([n]^r)=r$ for every $n\ge2$.

Suppose that \ref{char:uniform} holds.
For $P=[n]^r$, we have $\dim(Q\times[n]^r)\le r+s-r=s$ for every $n\ge2$.
Hence \ref{char:grid} holds.

Since $N\ge2$, condition~\ref{char:grid} implies \ref{char:finite}.

Suppose that \ref{char:finite} holds.
By Lemma~\ref{lem:extraction}, $Q$ has an $(r,s)$-covering.
Hence \ref{char:finite} implies \ref{char:cover}.
\end{proof}


\begin{thebibliography}{99}
\small
\raggedright
\setlength{\itemsep}{3pt}

\bibitem{Baker}
K. A. Baker,
Dimension, join-independence, and breadth in partially ordered sets,
unpublished manuscript, 1961.

\bibitem{Bergman1}
G. M. Bergman,
Some frustrating questions on dimensions of products of posets,
\emph{Discrete Math.} \textbf{349} (2026), no.~6, 115002.

\bibitem{Bergman2}
G. M. Bergman,
Further thoughts on dimensions of posets,
preprint, 2026,
\href{https://arxiv.org/abs/2607.13385}{arXiv:2607.13385}.

\bibitem{Blake1}
H. S. Blake, J. Hodor, P. Micek, M. T. Seweryn and W. T. Trotter,
Planarity and dimension I,
preprint, 2025,
\href{https://arxiv.org/abs/2510.18603}{arXiv:2510.18603}.

\bibitem{Blake2}
H. S. Blake, J. Hodor, P. Micek, M. T. Seweryn and W. T. Trotter,
Planarity and dimension II,
preprint, 2026,
\href{https://arxiv.org/abs/2607.09294}{arXiv:2607.09294}.

\bibitem{BMMV}
A. E. Brouwer, C. F. Mills, W. H. Mills and A. Verbeek,
Counting families of mutually intersecting sets,
\emph{Electron. J. Combin.} \textbf{20} (2013), no.~2, Paper~P8, 8~pp.

\bibitem{DW}
Z. Dong and K. Wang,
The Kelly--Trotter product conjecture for posets of dimension three,
preprint, 2026,
\href{https://arxiv.org/abs/2608.14434}{arXiv:2608.14434}.

\bibitem{DM}
B. Dushnik and E. W. Miller,
Partially ordered sets,
\emph{Amer. J. Math.} \textbf{63} (1941), 600--610.

\bibitem{Dushnik}
B. Dushnik,
Concerning a certain set of arrangements,
\emph{Proc. Amer. Math. Soc.} \textbf{1} (1950), no.~6, 788--796.

\bibitem{FMW}
S. Felsner, T. M\"utze and M. Wittmann,
Order dimension, grids, and products,
\emph{Order} \textbf{42} (2025), 811--827.

\bibitem{Hansel}
G. Hansel,
Sur le nombre des fonctions bool\'eennes monotones de $n$ variables,
\emph{C. R. Acad. Sci. Paris S\'er. A} \textbf{262} (1966), 1088--1090.

\bibitem{HM}
S. Ho\c{s}ten and W. D. Morris, Jr.,
The order dimension of the complete graph,
\emph{Discrete Math.} \textbf{201} (1999), 133--139.

\bibitem{Janzer}
B. Janzer,
A note on the orientation covering number,
\emph{Discrete Appl. Math.} \textbf{304} (2021), 349--351.

\bibitem{Joret}
G. Joret, P. Micek, M. Pilipczuk and B. Walczak,
Cliquewidth and dimension,
\emph{Proc. Lond. Math. Soc.} (3) \textbf{132} (2026), no.~1, e70116.

\bibitem{KT}
D. Kelly and W. T. Trotter,
Dimension theory for ordered sets,
in \emph{Ordered Sets},
I. Rival (ed.),
D. Reidel, Dordrecht, 1982, 171--211.

\bibitem{KnauerTrotter}
K. Knauer and W. T. Trotter,
Concepts of dimension for convex geometries,
\emph{SIAM J. Discrete Math.} \textbf{38} (2024), no.~2, 1566--1585.

\bibitem{KozikMicekTrotter}
J. Kozik, P. Micek and W. T. Trotter,
Dimension is polynomial in height for posets with planar cover graphs,
\emph{J. Combin. Theory Ser. B} \textbf{165} (2024), 164--196.

\bibitem{Reuter}
K. Reuter,
On the dimension of the Cartesian product of relations and orders,
\emph{Order} \textbf{6} (1989), 277--293.

\bibitem{Schnyder}
W. Schnyder,
Planar graphs and poset dimension,
\emph{Order} \textbf{5} (1989), 323--343.

\bibitem{ScottWood}
A. Scott and D. R. Wood,
Better bounds for poset dimension and boxicity,
\emph{Trans. Amer. Math. Soc.} \textbf{373} (2020), no.~3, 2157--2172.

\bibitem{Spencer}
J. Spencer,
Minimal scrambling sets of simple orders,
\emph{Acta Math. Acad. Sci. Hungar.} \textbf{22} (1971), 349--353.

\bibitem{Trotter}
W. T. Trotter,
The dimension of the Cartesian product of partial orders,
\emph{Discrete Math.} \textbf{53} (1985), 255--263.

\bibitem{Tsai}
S.-F. Tsai,
A simple upper bound on the number of antichains in $[t]^n$,
\emph{Order} \textbf{36} (2019), 507--510.

\bibitem{Wiedemann}
D. Wiedemann,
A computation of the eighth Dedekind number,
\emph{Order} \textbf{8} (1991), 5--6.

\end{thebibliography}
\end{document}